\documentclass[11pt]{article}
\usepackage{amssymb,amsmath,bm,amsthm,setspace}

\title{\bf Strong approximations in a charged-polymer model}
\author{Yueyun Hu\\Universit\'e Paris 13 \and   Davar Khoshnevisan\\University of Utah}
\date{November 19, 2009}

\renewcommand{\theequation}{\thesection.\arabic{equation}}
\newtheorem{theorem}{Theorem}[section]
\newtheorem{lemma}[theorem]{Lemma}
\newtheorem{proposition}[theorem]{Proposition}

\theoremstyle{definition}{}

\newcommand{\arrowfill}[1]{\ensuremath{\mathop{\hbox spread10pt{\rightarrowfill}}\limits^{#1}}}

\newcommand{\eqnsection}{
\renewcommand{\theequation}{\thesection.\arabic{equation}}
    \makeatletter
    \csname  @addtoreset\endcsname{equation}{section}
    \makeatother}
\eqnsection

\def\r{{\mathbf R}}
\def\e{{\mathrm E}}
\def\p{{\mathrm P}}

\def\z{{\mathbf Z}}

\def\law{\, {\buildrel{\rm law} \over =}\, }

\def\1{\mathbf{1}}

\def\as{\rm a.s.}
\def\qed{\hfill$\Box$}
\def\d{{\rm d}}

\begin{document}\onehalfspacing
\maketitle

\begin{abstract}
    We study the large-time behavior of the charged-polymer Hamiltonian
    $H_n$ of Kantor and Kardar [Bernoulli case] and Derrida, Griffiths, and Higgs
    [Gaussian case], using strong approximations to Brownian motion. Our results
    imply, among other things, that in one dimension the process $\{H_{[nt]}\}_{0\le t\le 1}$
    behaves like  a Brownian motion, time-changed by the intersection
    local-time process of an independent Brownian motion. Chung-type LILs
    are also discussed.\\

    \noindent{\it Running Title:} Polymer measures, strong approximations.
\end{abstract}

\section{Introduction}

Consider  a sequence $\{q_i\}_{i=1}^\infty$  of independent, identicallly-distributed
mean-zero random variables, and let $S:=\{S_i\}_{i=0}^\infty$
denote an independent simple random walk on $\z^d$ starting from $0$.  For $n\ge 1$, define
\begin{equation}
    H_n:= \mathop{\sum\sum}_{1\le i < j \le n} q_i q_j \1_{\{S_i=S_j\}};
\end{equation}
this is the Hamiltonian of a socalled ``charged polymer model.'' See Kantor and Kadar
\cite{KK91} in the case that the $q_i$'s are Bernoulli, and Derrida,
Griffiths, and Higgs \cite{DGH92} for the case of Gaussian random variables.
Roughly speaking, $q_1,q_2,\ldots$ are random charges that
are placed on a polymer path modeled by the trajectories of $S$;
and one can construct a  Gibbs-type polymer measure from the Hamiltonian $H_n$.

We follow Chen \cite{C08} (LIL and moderate deviations), Chen and
Khoshnevisan \cite{CK09} (comparison between $H_n$ and the random
walk in random scenery model), and Asselah \cite{A09+} (large deviations
in high dimensional case), and continue the analysis of the
Hamiltonian $H_n$. We assume here and in the sequel that
\begin{equation}
    \e(q_1^2)=1\ \text{and}\
    \e \left(|q_1|^{p(d)}\right) <\infty \ \text{ where }\ p(d):=
    \begin{cases}
        6&\text{if $d=1$},\\
        4&\text{if $d\ge 2$}.
    \end{cases}
\end{equation}

\begin{theorem}\label{T1}
    On a possibly-enlarged probability space, we can
    define a version of $\{H_n\}_{n=1}^\infty$ and a one-dimensional Brownian motion
    $\{\gamma(t) \}_{t\ge0}$ such that     the following holds almost surely:
    \begin{equation}
        H_n = \begin{cases}\displaystyle
            \frac{1}{\sqrt 2}\,\gamma\left(\int_{-\infty}^\infty(\ell_n^x)^2\,\d x
            \right) + o ( n^{\frac34-\epsilon} )&\text{if $d=1$ and $0<\epsilon< {1\over 24}$},\\[3mm]\displaystyle
        \frac{1}{\sqrt{2\pi}}\,\gamma(n\log n) + O( n^{\frac12} \log \log n)&\text{if $d=2$},\\[3mm]
            \sqrt\kappa\,\gamma(n)+o (n^{\frac12-\epsilon} )&\text{if $d\ge 3$ and  $0<\epsilon< {1\over 8}$},
        \end{cases}
    \end{equation}
    where $\{\ell_t^x \}_{t\ge 0,x\in\r}$ denotes the local times of a
    linear Brownian motion $B$ independent of $\gamma$,
    and $\kappa:=\sum_{k=1}^\infty\p\{S_k=0\}$.
\end{theorem}


It was shown in \cite{CK09} that when $d=1$ the distribution of $H_n$
converges, after normalization, to the ``random walk in random scenery.''
The preceding shows that the stochastic process $\{H_{[nt]}\}_{0\le t\le 1}$
does {\it not} converge weakly to the random walk in random scenery; rather,
we have the following consequence of Brownian scaling for all $T>0$: As $n\to\infty$,
\begin{equation}
    \left\{ \frac{H_{[nt]}}{n^{3/4}}\right\}_{0\le t\le T}
    \arrowfill{D([0,T])}\, \left\{ \frac{1}{\sqrt 2}\, \gamma\left(
    \int_{-\infty}^\infty (\ell^x_t)^2\,\d x\right) \right\}_{0\le t\le T}.
\end{equation}

With a little bit more effort,
we can also obtain strong limit theorems. Let us state the following
counterpart to the LILs of Chen \cite{C08}, as it appears to have novel
content.

\begin{theorem}\label{T2}  Almost surely:
(i) If $d=1$,  then
    \[
        \liminf_{n \to\infty} \left({ \log \log n \over n}
        \right)^{3/4} \max_{0\le k \le n} |H_k | = { (a^*)^{3/4}
        \frac\pi 4},
    \]
    where $a^*= 2.189\pm 0.0001$ is a numerical constant
    {\rm \cite[(0.6)]{HHK97}};

    (ii) If $d=2$, then
    \[
        \liminf_{n \to\infty} \sqrt{ \frac{\log \log n}{n\log n}}
        \max_{0\le k \le n} |H_k | = {\frac{\sqrt{\pi}}{4}};
    \]

    (iii) If $d\ge3$,  then
    \[
        \liminf_{n \to\infty} \sqrt{ \frac{\log \log n}{n}}
        \max_{0\le k \le n} |H_k | =  \pi \sqrt{\frac{\kappa}{8}},
    \]
    where $\kappa$ was defined in Theorem \ref{T1}.
\end{theorem}

Theorems \ref{T1} and \ref{T2} are proved respectively in Sections 2 and 3.

\section{Proof of Theorem \ref{T1}}

Let $W$ be a one-dimensional Brownian motion starting from $0$. By the
Skorohod embedding  theorem, there exists a sequence of stopping times
$\{T_n\}_{n=1}^\infty$ such that $\{T_n - T_{n-1}\}_{n=1}^\infty$ (with $T_0=0$)
are i.i.d., and:
\begin{equation}\begin{split}
    &\e(T_1)=\e(q_1^2)=1,\quad
        \text{Var}(T_1)\le\text{const}\cdot\e(q_1^4)<\infty,\quad
        \text{and}\\
    &\quad\left\{ W(T_n) - W(T_{n-1})\right\}_{n=1}^\infty \law
        \{q_n\}_{n=1}^\infty.
\end{split}\end{equation}

Throughout this paper, we take the following special construction
of the charges $\{q_i\}_{i=1}^\infty$:
\begin{equation}
    q_n:=W(T_n) - W(T_{n-1})
    \qquad\text{for $n\ge 1$}.
\end{equation}
Next, we describe how we
choose a special construction of the random walk $S$, depending on $d$.

If $d=1$, then on a possibly-enlarged probability space let $B$ be
another one-dimensional Brownian motion, independent of $W$. By
using a theorem of R\'ev\'esz \cite{R81}, we may construct a
one-dimensional simple symmetric random walk $\{S_i\}_{i=1}^\infty$ from $B$ such
that almost surely,
\begin{equation}\label{R1}
    \sup_{x\in \z } \left| L_n^x - \ell_n^x  \right| =
    n^{\frac14+o(1)}
    \quad\text{as $n\to\infty$,\ where}\
    L_n^x:=\sum_{i=1}^n \1_{\{ S_i=x\}},
\end{equation}
and $\ell_n^x $ denotes the local times of
$B$ at $x$ up to time $n$.

If $d\ge2$, then we just choose   an independent simple symmetric  random
walk $\{S_n\}_{n=1}^\infty$, after enlarging the probability space, if we need to.

Now we define the Hamiltonians $\{H_n\}_{n=1}^\infty$ via the preceding
constructions of $\{ q_i\}_{i=1}^\infty$
and $\{S_n\}_{n=1}^\infty$. That is,
\begin{equation}\begin{split}
    H_n &= \mathop{\sum\sum}_{1\le i< j \le n} ( W(T_i) -W(T_{i-1})) ( W(T_j) - W(T_{j-1}))
        \1_{(S_i=S_j)} \\
    &= \int_0^{T_n} G_n\,\d W,
\end{split}\end{equation}
where, for all integers $n\ge 1$ and reals $s\ge 0$,
\begin{equation}
    G_n(s) := \mathop{\sum\sum}_{1\le i<j\le n}
    \1_{( S_i= S_j)} (W(T_i) - W( T_{i-1})) \1_{( T_{j-1} \le s < T_j)}.
\end{equation}

By the Dambis, Dubins--Schwarz representation theorem \cite[Theorem 1.6, p.\ 170]{RY},
after possibly enlarging the underlying probability space, we can find a
one-dimensional Brownian motion $\gamma$ such that
$\int_0^t G_n\,\d W$ is equal to $\gamma  (
\int_0^t |G_n(s)|^2\,\d s )$
for $t\ge 0$.
We stress the fact that if $d=1$, then   {\it$\gamma$ is
independent} of $B$. This is so, because the bracket between the two continuous
martingales vanishes:
$\langle\int_0^\bullet G_n\,\d W\,,\,B
\rangle_t=0$ for $t\ge 0$.
Consequently, the following holds for all $n\ge 1$: Almost surely,
\begin{equation}\label{H:Xi}
    H_n= \gamma(\Xi_n),
    \quad\text{where}\quad
    \Xi_n  := \int_0^{T_n} |G_n(s)|^2\,\d s.
\end{equation}

\begin{proposition}\label{pr:main}
    The following holds almost surely:
    \begin{equation}\begin{split}
        \Xi_n = \begin{cases}\displaystyle
            \dfrac12\int_{-\infty}^\infty
            (\ell_n^x)^2\,\d x + O ( n^{\frac32-\epsilon} )&\text{if $d=1$ and $ 0 < \epsilon < \frac{1}{12}$},\\[2mm]
            \dfrac{1}{2\pi}\, n\log n + O(n\log\log n)&\text{if $d=2$},\\[2mm]
            \kappa n+O(n^{1-\epsilon})&\text{if $d\ge 3$ and $ 0 < \epsilon < \frac14$}.
        \end{cases}
    \end{split}\end{equation}
\end{proposition}

We prove this proposition later. First, we show that in case $d=1$,
the preceding proposition estimates $\Xi_n$ correctly to leading term.

\begin{lemma}\label{lem:alpha}
    If $d=1$, then a.s.,
    $\int_{-\infty}^\infty (\ell_n^x )^2\,\d x = n^{\frac32+o(1)}$ as $n\to\infty$.
\end{lemma}

\begin{proof}
    This is well known; we include a proof for the sake of completeness.

    Because $\int_{-\infty}^\infty\ell_n^x\,\d x=n$,
    we have $\int_{-\infty}^\infty(\ell_n^x)^2\,\d x  \le n\sup_{-\infty<x<\infty} \ell_n^x$,
    and this is $n^{\frac32+o(1)}$ \cite{K65}.
    For the converse bound we apply the
    Cauchy--Schwarz inequality to find that
    $n^2 =( \int_{-\infty}^\infty \ell_n^x\,\d x)^2
    \le \int_{-\infty}^\infty
    (\ell_n^x)^2\,\d x \cdot \text{Osc}_{[0,n]}B,$
    where $\text{Osc}_{[0,n]}B := \sup_{[0,n]}B-\inf_{[0,n]}B
    =n^{\frac12+o(1)}$ by Khintchine's LIL. This completes the proof.
\end{proof}

Let us complete the proof of Theorem \ref{T1}, first assuming Proposition
\ref{pr:main}. That proposition will then be proved subsequently.

\begin{proof}[Proof of Theorem \ref{T1}]

    We shall consider only the case $d=1$;
    the cases $d=2$ and $d\ge 3$ are proved similarly.
    We apply  the Cs\"org\H{o}--R\'ev\'esz modulus of continuity
    of Brownian motion \cite[Theorem 1.2.1]{CR81}
    to $H_n=\gamma(\Xi_n)$---see \eqref{H:Xi}---with
    the changes of variables, $t=n^{\frac32+o(1)}$ and
    $a(t)=n^{\frac32-\epsilon}$;
    then apply Lemma \ref{lem:alpha} to see that
    $|\gamma(\Xi_n)-\gamma (\frac12\int_{-\infty}^\infty
    (\ell_n^x)^2\,\d x ) | = O(n^{\frac34-\epsilon})$ as $n\to\infty$ a.s.
\end{proof}

\begin{lemma}\label{lem:main}
    The following holds almost surely:
    \begin{equation}\begin{split}
        \mathop{\sum\sum}_{1\le i<k\le n} \1_{\{S_i=S_k\}}
        = \begin{cases}\displaystyle
            \dfrac12\int_{-\infty}^\infty
            (\ell_n^x)^2\,\d x + n^{\frac54+o(1)}&\text{if $d=1$},\\[2mm]
            \dfrac{1}{2\pi}\, n\log n + O(n\log\log n)&\text{if $d=2$},\\[2mm]
            \kappa n+n^{\frac12+o(1)}&\text{if $d\ge 3$}.
        \end{cases}
    \end{split}\end{equation}
\end{lemma}

\begin{proof}
    In the case that $d=2$, this result follows from Bass, Chen and Rosen \cite{BCR06};
    and in the case $d\ge 3$, from Chen \cite[Theorem 5.2]{C08}.
    Therefore, we need to only check the case  $d=1$.

    We begin by writing
    $\mathop{\sum\sum}_{1\le i< k\le n} \1_{\{S_i=S_k\}}
    =  \frac12 \sum_{x\in \z}
    (L_n^x)^2 - \frac n2$.
    According to Bass and Griffin \cite[Lemma 5.3]{BG85},
    $\sup_{x\in \z}\, \sup_{y\in[x,x+1]} |\ell_n^x -
    \ell_n^y|= n^{\frac 14+o(1)}$ as $n\to\infty$ a.s.
    This and \eqref{R1} together imply that
    $| L_n^x - \ell_n^y |= n^{\frac14+o(1)}$ uniformly over all $y\in[x\,,x+1]$
    and $x\in\z$ [a.s.],
    whence
    \begin{equation}\begin{split}
        \left| \sum_{x\in \z} (L_n^x)^2 - \int_{-\infty}^\infty (\ell^y_n)^2
            \,\d y \right|
        &\le n^{\frac14+o(1)} \cdot \sum_{x\in \z} \int_x^{x+1}   (L_n^x +
            \ell^y_n) \,\d y.
    \end{split}\end{equation}
    Since the latter sum is equal to $2n$, the lemma follows.
\end{proof}

\begin{lemma}\label{lem:1}
    The following holds a.s.: As $n\to\infty$,
    \begin{equation}\begin{split}
        \mathop{\sum\sum}_{1\le i<k\le n}
        \1_{\{S_i=S_k\}} \left(q_i^2-1\right) = \begin{cases}
            n^{\frac76 + o(1)} &\text{if $d=1$},\\[2mm]
            n^{\frac34+ o(1)} &\text{if $d\ge 2$}.
        \end{cases}
    \end{split}\end{equation}
\end{lemma}

\begin{proof}
    We can let $M_n$ denote the double sum in the lemma, and check directly that
    $M_n = \sum_{1\le i\le n-1}
    ( L^{S_i}_n-L^{S_i}_i)(q_i^2-1)$.
    Let $\mathcal{S}$ denote the $\sigma$-algebra generated by the
    entire process $S$. Then, conditionally on $\mathcal{S}$,
    each $M_n$ is a sum of independent  random variables.
    By Burkholder's inequality \cite[Theorem 2.10, p.\ 34]{HH},
    for all even integers $p\ge 2$,
    \begin{equation}
        \e\left(|M_n|^p\right)
        \le \text{const}\cdot
        \e\left(\left| \sum_{1\le i\le n-1}\left(L_n^{S_i}-L_i^{S_i}\right)^2
        (q_i^2-1)^2\right|^{p/2}\right).
    \end{equation}
    According to the generalized H\"older inequality,
    \begin{equation}
        \e\left(\prod_{k=1}^{p/2}(q_{i_k}^2-1)^2\right)
        \le \prod_{k=1}^{p/2}\left\{ \e\left(\left| q_{i_k}^2-1\right|^p\right)\right\}^{2/p}
        = \e\left(|q_1^2-1|^p\right).
    \end{equation}
    Another application of the generalized H\"older inequality,
    together with an appeal to the Markov property, yields
    \begin{equation}\begin{split}
        \e\left(\prod_{k=1}^{p/2}\left(L_n^{S_{i_k}}-L_{i_k}^{S_{i_k}}\right)^2\right)
            &\le \prod_{k=1}^{p/2}\left\{ \e\left( \left| L_n^{S_{i_k}}-L_{i_k}^{S_{i_k}}
            \right|^p\right)\right\}^{2/p}\\
        &=\prod_{k=1}^{p/2}\left\{\e\left(\left| L_{n-i_k}^0\right|^p\right)
            \right\}^{2/p}.
    \end{split}\end{equation}
    Therefore, we can apply the local-limit theorem to find that
    \begin{equation}\begin{split}
        \e\left(|M_n|^p\right)
        &\le \text{const}\cdot\begin{cases}
                n^p&\text{if $d=1$},\\
                n^{\frac p2+o(1)}&\text{if $d\ge 2$}.
            \end{cases}
    \end{split}\end{equation}
    The lemma follows from this and the Borel--Cantelli lemma.
\end{proof}

\begin{lemma}\label{lem:2}
    The following holds almost surely: As $n\to\infty$,
    \begin{equation}\begin{split}
        \mathop{\sum\sum}_{1\le i<k\le n}
        \1_{\{S_i=S_k\}} q_i^2 \left(T_k-T_{k-1}-1\right) = \begin{cases}
            n^{1+o(1)}&\text{if $d=1$},\\[2mm]
            n^{\frac12+o(1)}&\text{if $d\ge 2$}.
        \end{cases}
    \end{split}\end{equation}
\end{lemma}
\begin{proof}
    Let $N_n$ denote the double sum in the lemma, and  note
    that
    \begin{equation}\begin{split}
        N_n =& \sum_{2\le k\le n}\beta_{k-1} \left(T_k-T_{k-1}-1\right),
            \ \text{where}\\
        \beta_{k-1}:=& \sum_{1\le i\le k-1} q_i^2 \1_{\{S_i=S_k\}}.
    \end{split}\end{equation}
    Recall that $\mathcal{S}$ denotes the $\sigma$-algebra generated by the
    entire process $S$ and observe that, conditionally on $\mathcal{S}$,
    $\{N_n\}_{n=1}^\infty$ is a mean-zero martingale with
    \begin{equation}\begin{split}
        \e(N_n^2\, |\, \mathcal{S}) &= \text{Var}(T_1)\cdot
            \sum_{2\le k\le n}\e\left( \beta_{k-1}^2\, \big|\, \mathcal{S}\right)\\
        &=\text{Var}(T_1)\cdot \sum_{2\le k\le n}\left(
            \text{Var}(q_1^2) + \left| \e(q_1^2)\right|^2
            \cdot L_{k-1}^{S_k} \right) L_{k-1}^{S_k}.
    \end{split}\end{equation}
    This and Doob's inequality together show that
    \begin{equation}
        \e\left(\max_{1\le k\le n} N_k^2\right) \le\text{const}\cdot n
        \left( \max_{a\in\z}\e (L^a_n) + \max_{a\in\z}\e(|L^a_n|^2)
        \right).
    \end{equation}
    By the local-limit theorem, the preceding is at most
    a constant multiple of $n (\sum_{1\le i\le n} i^{-d/2} )^2$.
    The Borel--Cantelli lemma finishes the proof.
\end{proof}

\begin{proof}[Proof of Proposition \ref{pr:main}]
    Recall the definition of each $q_i$. With that in mind,
    we can decompose $\Xi_n$ as follows:
    \begin{equation}
        \Xi_n = \sum_{1\le k\le n}\int_{T_{k-1}}^{T_k}
        \d s\left( \mathop{\sum}_{1\le i< k} \1_{\{ S_i= S_k\}}
        q_i \right)^2
        = \Xi_n^{(1)} + \Xi_n^{(2)},
    \end{equation}
    where,
    \begin{equation}
        \Xi^{(1)}_n := \mathop{\sum\sum}_{1\le i< k\le n} \1_{\{S_i=S_k\}}
        q_i^2(T_k-T_{k-1}),
    \end{equation}
    and
    \begin{equation}
        \Xi^{(2)}_n := 2 \mathop{\sum\sum\sum}_{1\le i<j< k\le n}
        \1_{\{S_i=S_j=S_k\}} q_iq_j(T_k-T_{k-1}).
    \end{equation}
    Since
    \begin{equation}\begin{split}
        \Xi^{(1)}_n &=\mathop{\sum\sum}_{1\le i< k\le n} \1_{\{S_i=S_k\}} +
            \mathop{\sum\sum}_{1\le i< k\le n} \1_{\{S_i=S_k\}} \left(q_i^2-1\right)\\
        &\hskip1.5in+
            \mathop{\sum\sum}_{1\le i< k\le n} \1_{\{S_i=S_k\}} q_i^2\left(T_k-T_{k-1}-1\right),
    \end{split}\end{equation}
    Lemmas \ref{lem:alpha},
    \ref{lem:main}, \ref{lem:1}, and \ref{lem:2} together imply
    that $\Xi^{(1)}_n$ has the large-$n$ asymptotics that is
    claimed for $\Xi_n$. In light of Lemma \ref{lem:alpha},
    it suffices to show that  almost surely  the
    following holds  as $n\to\infty$:
    \begin{equation}\label{eq:Xi:2}
        \Xi^{(2)}_n =\begin{cases}
            n^{\frac{17}{12} + o(1)}&\text{if $d=1$ },\\
            n^{\frac34+o(1)}&\text{if $d\ge 2$}.
        \end{cases}
    \end{equation}

    We can write
    $\Xi^{(2)}_n  =2\sum_{k=3}^n \tau_{k-1}(S_k) (T_k-T_{k-1} )$,
    where
    \begin{equation}
        \tau_{k-1}(z):= \mathop{\sum\sum}_{1\le i<j\le k-1} \1_{\{
        S_i=S_j=z\}}q_iq_j
        \qquad\text{for $z\in\z$ and $k> 1$}.
    \end{equation}
    In particular, we can write
    \begin{equation}
        \Xi^{(2)}_n := 2\left(a_n + b_n\right),
        \qquad\text{where}
    \end{equation}
    \begin{equation}
        a_n := \sum_{k=3}^n \tau_{k-1}(S_k)\left(T_k-T_{k-1}-1\right)
        \quad\text{and}\quad
        b_n := \sum_{k=3}^n \tau_{k-1}(S_k).
    \end{equation}

    Recall that $\mathcal{S}$ denotes the $\sigma$-algebra generated by
    the process $S$. It follows that, conditional on $\mathcal{S}$,
    the process $\{a_n\}_{n=1}^\infty$ is a mean-zero martingale, and
    \begin{equation}
        \e \left( a_n^2 \,  \big|\, \mathcal{S} \right)
        = \text{Var}(T_1)\cdot\sum_{k=3}^n \e \left( |\tau_{k-1}(S_k)|^2\,
        \big|\, \mathcal{S} \right).
    \end{equation}
    The latter conditional expectation is also computed by a martingale
    computation. Namely,
    we write $\tau_{k-1}(z) = \sum_{j=2}^{k-1} (
    \sum_{i=1}^{j-1} \1_{\{S_i=S_j=z\}}q_i )q_j$
    for all $z\in\z$    in order to deduce that
    \begin{equation}
        \e\left( |\tau_{k-1}(z)|^2\,\big|\, \mathcal{S}\right)
            = \sum_{j=2}^{k-1} \1_{\{S_j=z\}}
            L^z_{j-1}
        \le \left( L^z_{k-1}\right)^2.
    \end{equation}
    It follows from Doob's maximal inequality that
    \begin{equation}
        \e \left( \max_{1\le k\le n}a_k^2 \right)
        \le 4\text{Var}(T_1)\cdot\e\left(
        \sum_{k=3}^n (L^{S_k}_{k-1})^2\right).
    \end{equation}
    By time reversal, we can replace $L^{S_k}_{k-1}$ by $L^0_{k-1}$.
    Therefore, the local-limit theorem implies that
    $\e  ( \max_{1\le k\le n}a_k^2  )
    \le \text{const}\cdot n (\sum_{i=1}^n i^{-d/2} )^2$,
    and hence almost surely as $n\to\infty$,
    \eqref{eq:Xi:2} is satisfied with $\Xi^{(2)}_n$ replaced by
    $a_n$ [the Borel--Cantelli lemma].
    It suffices to prove that \eqref{eq:Xi:2} holds if
    $\Xi_n$ is replaced by $b_n$.

    We can write $b_n:=b_{n,n}$, where
    \begin{equation}
        b_{n,k} =\sum_{j=2}^{n-1} \theta_{j-1,k}\,q_j
        \quad\text{for}\quad
        \theta_{j-1,k} :=\sum_{i=1}^{j-1} \1_{\{S_i=S_j\}}
        q_i\left( L^{S_i}_k - L^{S_i}_j\right).
    \end{equation}
    For each fixed integer $k\ge 1$, $\{b_{n,k}\}_{n\ge 3}$ is a mean-zero martingale,
    conditional on $\mathcal{S}$. Therefore, Burkholder's inequality yields
    \begin{equation}
        \e\left(|b_{n,k}|^p\,\big|\,\mathcal{S}\right) \le
        \text{const}\cdot\e\left(\left.\left[
        \sum_{j=2}^{n-1}\theta_{j-1,k}^2\,q_j^2
        \right]^{p/2}\, \right|\ \mathcal{S}\right),
    \end{equation}
    where the implied constant is nonrandom and depends only on $p$.
    Since $|\sum_{j=1}^{n-1} x_j|^{p/2}\le n^{\frac p2-1}\sum_{j=1}^{n-1}|x_j|^{p/2}$
    for all real $x_1,\ldots,x_{n-1}$, we can apply the preceding with $k:=n$ to obtain
    \begin{equation}
        \e\left(|b_n|^p\,\big|\,\mathcal{S}\right) \le
        \text{const}\cdot \e(|q_1|^p)\cdot
        n^{\frac p2-1} \sum_{j=2}^{n-1}
        \e\left( |\theta_{j-1,n}|^p\, \big|\ \mathcal{S}\right).
    \end{equation}
    Yet another application of Burkholder's inequality yields
    \begin{equation}\begin{split}
        \e\left(|\theta_{j-1,n}|^p\,\big|\ \mathcal{S}\right)
            &\le \text{const}\cdot\e\left(\left| \sum_{i=1}^{j-1}
            \1_{\{S_i=S_j\}}q_i^2
            \left(L^{S_i}_n-L^{S_i}_j\right)^2
            \right|^{p/2}\right)\\
        &\le\text{const}\cdot\e(|q_1|^p)\cdot\left( L_{j-1}^{S_j}
            \right)^{p/2}\left(L^{S_j}_n-L^{S_j}_j\right)^p,
    \end{split}\end{equation}
    since $\e(q_{i_1}^2\cdots q_{i_{p/2}}^2)\le\e(|q_1|^p)$ for all
    $1\le i_1,\ldots,i_{p/2}<j$. We take expectations and apply the
    Markov property and time reversal to find that
    \begin{equation}
        \e\left(|\theta_{j-1,n}|^p\right)
        \le \text{const}\cdot\e(|q_1|^p)\cdot
        \e\left[\left( L^0_{j-1}\right)^{p/2}\right]
        \e\left[\left( L^0_{n-j}\right)^p\right].
    \end{equation}
    It follows readily that
    \begin{equation}
        \e\left(|b_n|^p\right) \le \text{const}\cdot \e(|q_1|^p)\cdot
        n^{p/2}\e\left[\left( L^0_n\right)^{p/2}\right]
        \e\left[\left( L^0_n\right)^p\right]\\
    \end{equation}
    This, the local-limit theorem, and the Borel--Cantelli
    lemma together imply that \eqref{eq:Xi:2} holds with $b_n$ in place
    of $\Xi^{(2)}_n$. The proposition follows.
\end{proof}

\section{Proof of Theorem \ref{T2}}

In  view of Theorem \ref{T1} and the LIL for the Brownian motion,
it suffices to consider only the
case $d=1$, and to establish the following:
\begin{equation}\label{LIL}
    \liminf_{n \to\infty} \left(\frac{\L_2 n}{n}\right)^{3/4}
    \max_{0\le k \le n} |\gamma(\alpha(k)) | = (a^*)^{3/4}
    \frac{\pi}{\sqrt{8}},
\end{equation}
where
\begin{equation}
    \L_2 x:=\log\log(x\vee 1)
    \quad\text{and}\quad
    \alpha(t) := \int_{-\infty}^\infty (\ell^x_t)^2\,\d x.
\end{equation}
 It is known that \cite[Theorem 3]{CCFR83},
 $\sup_{x\in\r} \sup_{1\le k\le n} (\ell^x_k - \ell^x_{k-1}) =o((\log n)^{-1/2})$
 almost surely $[\p]$. This implies readily that
$\max_{0\le k \le n} (\alpha(k+1) -
\alpha(k-1)) = O(n \sqrt{\log n})$ a.s. Therefore, it follows from
\cite[Theorem 1.2.1]{CR81}  that  \eqref{LIL} is equivalent to the following:
\begin{equation}\label{LIL2}
    \liminf_{t \to\infty} \left(\frac{\L_2 t}{t} \right)^{3/4}
    \sup_{0\le s \le t} |\gamma(\alpha(s)) | =
    (a^*)^{3/4} \frac{\pi}{\sqrt{8}}.
\end{equation}

Brownian scaling implies that $\alpha(t)$ and
$t^{3/2} \alpha(1)$ have the same distribution.
On one hand, Proposition 1 of \cite{HHK97} tells us that
the limit $C:=\lim_{\lambda \to \infty} \exp\{a^*\, \lambda^{2/3}\}\,  \e
\exp( -\lambda \alpha(1))$
exists and is positive and finite.
On the other hand, we can
write $\gamma^*(t):=\sup_{0\le s\le t} |\gamma(s)|$ and appeal
to Lemma 1.6.1 of
\cite{CR81} to find that for all $t, y>0$,
\begin{equation}
    \frac2\pi \exp\left(- \frac{\pi^2 t}{8y^2}\right) \le
    \p\left\{ \gamma^*(t) < y \right\}\le \frac4\pi \exp\left(- \frac{\pi^2 t}{8y^2}\right).
\end{equation}
Therefore uniformly for all $t>0$ and $x\in(0\,,1]$,
\begin{equation}\label{small}\begin{split}
    \p\left\{\sup_{0\le s \le t } |\gamma(\alpha(s)) | < x t^{3/4}
        \right\} &=   \p\left\{\gamma^*(\alpha(1)) <
        x\right\}
        \le  {4\over \pi}  \e\,{\rm e}^{-\pi^2\alpha(1)/(8x^2)}\\
    &\hskip.4in\le  \text{const}\cdot
        \exp\left( -\frac{a^*}{x^{4/3}}\left( \frac{\pi^2}{8} \right)^{2/3}\right).
\end{split}\end{equation}
This and an application of the Borel--Cantelli lemma together yield
one half of the \eqref{LIL2}; namely, \eqref{LIL2} where ``='' is
replaced by ``$\ge$.''
In order to derive the other half we choose $t_n := n^n$ and $c>
(a^*)^{3/4} \pi /\sqrt 8$, and define
\begin{equation}
    A_n :=\left\{\omega:\,
    \sup_{0\le s \le t_n } |\gamma(\alpha(s)) | < c  \left(
    \frac{t_n}{\L_2 t_n}\right)^{3/4} \right\}.
\end{equation}
Every $A_n$ is measurable with respect to ${\cal
F}_{t_n}:= \sigma\{\gamma(u):\,u\le \alpha(t_n)\}
\vee\sigma\{B_v:\, v\le t_n\}$.
In light of the 0-1 law of Paul L\'evy, and since $c>(a^*)^{3/4} \pi /\sqrt 8$
is otherwise arbitrary, it suffices to prove that
\begin{equation}\label{levy}
    \sum_{n=1}^\infty \p\left(\left. A_n \ \right|\, {\cal F}_{t_{n-1}}\right)=\infty
    \qquad\as
\end{equation}

The argument that led to \eqref{small} can be
used to show that for all $v\ge 0$,
\begin{align}\label{small2}
    \p\left\{ \gamma^*(v+\alpha(t)) < x t^{3/4}\right\}
        &\ge \frac2\pi \exp\left( -\frac{\pi^2 v}{8 x^2 t^{3/2}}\right)\,
        \e\,{\rm e}^{-\pi^2\alpha(1)/(8x^2)}\\\nonumber
      &\ge \text{const}\cdot    \exp\left(
        - {\pi^2 v\over 8 x^2 t^{3/2}} - \frac{a^*}{x^{4/3}} \left( \frac{\pi^2}{8}
        \right)^{2/3} \right).
\end{align}

In order to prove \eqref{levy}, let us choose and fix a large integer $n$
temporarily. We might note that
\begin{equation}
    \sup_{0\le s \le t_n } |\gamma(\alpha(s)) |
    = \gamma^*(\alpha(t_n)) \le  \gamma^*(\alpha(t_{n-1})) +
    \widetilde\gamma^* (\alpha(t_n) - \alpha(t_{n-1})),
\end{equation}
where $\widetilde\gamma(s) = \widetilde\gamma_n(s):= \gamma(s+ \alpha(t_{n-1})) - \gamma(\alpha(t_{n-1}))$
and $\widetilde\gamma^*(s):= \sup_{0\le v\le s} |\widetilde\gamma(v)|$ for $s\ge 0$.
Of course, $\widetilde\gamma$ is a Brownian motion independent of ${\cal F}_{t_{n-1}}$.
Moreover, we can write
$\ell_{t_n}^x = \ell_{t_{n-1}}^x + \widetilde \ell_{t_n - t_{n-1}}^{x- B_{t_{n-1}}}$,
where $\widetilde \ell$ denotes the local time process of the
Brownian motion $\widetilde B(s):= B(s+ t_{n-1}) - B(t_{n-1}), s\ge 0$.
Clearly, $(\widetilde \gamma\,, \widetilde B)$ is a two-dimensional Brownian motion,
independent of ${\cal F}_{t_{n-1}}$.  Observe that
\begin{align}\nonumber
    \alpha_{t_n} - \alpha_{t_{n-1}} &=
        \int \d x  \left[ (\ell_{t_n}^x )^2 -  (\ell_{t_{n-1}}^x )^2 \right]
        =  \int \d x\ \widetilde \ell_{t_n- t_{n-1}}^{ x - B_{t_{n-1}}}
        \left( \ell_{t_{n-1}}^x + \widetilde \ell_{t_n- t_{n-1}}^{ x - B_{t_{n-1}}}
        \right)\\
    &\le (t_n - t_{n-1}) \ell^*_{t_{n-1}} + \widetilde \alpha_{t_n - t_{n-1}},
\end{align}
where $\ell^*_{t_{n-1}}:= \sup_{x\in \r} \ell^x_{t_{n-1}}$. Therefore, we obtain
\begin{equation}
    \gamma^*(\alpha_{t_n}) \le  \gamma^*(\alpha_{t_{n-1}}) +
    \widetilde \gamma^* \left((t_n - t_{n-1}) \ell^*_{t_{n-1}} +
    \widetilde \alpha_{t_n - t_{n-1}}\right).
\end{equation}

Let $\varepsilon>0$ be such that $2\varepsilon< c-(a^*)^{3/4} \pi / \sqrt{8}$,
and define
\begin{equation}
    D_n:= \left\{ \gamma^*(\alpha_{t_{n-1}}) < \varepsilon
    \left({t_n \over \L_2  t_n}\right)^{3/4},\, \ell^*_{t_{n-1}}
    \le \sqrt{3t_{n-1 }\L_2  t_{n-1 }  }  \right\}.
\end{equation}
Clearly, $D_n$ is ${\cal F}_{t_{n-1}}$-mesurable.
Let $v_n:= t_n \sqrt{3t_{n-1} \L_2 t_{n-1  } }$. Since
\begin{equation}
    A_n \supset D_n \cap \left\{ \widetilde \gamma^*
    (v_n + \widetilde \alpha_{t_n  }) \le (c-\varepsilon) \left({t_n \over \L_2  t_n}
    \right)^{3/4}  \right\},
\end{equation}
we can deduce from \eqref{small2} that
\begin{align}\nonumber
    &\p ( A_n \,  |\, {\cal F}_{t_{n-1}} ) \\\nonumber
    &\ge \text{const} \cdot \1_{D_n} \exp\left(
        - \frac{\pi^2 v_n}{8}\left(\frac{\L_2  t_n}{t_n}\right)^{3/2}  -
        a^* \left(\frac{\pi^2}{8(c-\varepsilon)^2}\right)^{2/3} \L_2  t_n\right)\\
    &\ge  \text{const}\cdot
        \1_{D_n} \, (n\ln n)^{ - a^* (\pi^2/8)^{2/3} (c-\varepsilon)^{-4/3} },
\end{align}
where we have used the fact that  $v_n /t_n^{3/2} \sim 1/n$.
Because $a^* (\pi^2/8)^{2/3} (c-\varepsilon)^{-4/3} < 1$,
\eqref{levy} implies that almost surely, $\1_{D_n}=1$ for all $n$ large.
Indeed, the LIL tells us that almost surely for all large $n$,
$\ell^*_{t_{n-1}} \le \sqrt{3t_{n-1 }\L_2 t_{n-1 }  } $,
and $\gamma^*(\alpha_{t_{n-1}}) \le \sqrt{ 3 \alpha_{t_{n-1}}
\L_2 \alpha_{t_{n-1}}}$. Since $\alpha(t) =\int (\ell_t^x)^2 dx \le t\, \ell^*_t $, we
find that $\alpha_{t_{n-1}} \le t_{n-1}^{3/2} \sqrt{3 \L_2 t_{n-1}}$.
Since $t_{n-1}/t_n \sim 1/n$, it follows that almost surely,
$\ell^*_{t_{n-1}} \le  \varepsilon (t_n / \L_2 t_n)^{3/4} $
for all large $n$ and prove that $D_n$ realizes eventually for all large
$n$. The proof of Theorem \ref{T2} is complete. \qed\\

\noindent\textbf{Acknowledgements.} We thank the University of Paris
for generously hosting the second author (D.K.), during which
time much of this research was carried out.

\begin{small}

\vskip1cm

\noindent
\textbf{Yueyun Hu.}  D\'epartement de Math\'ematiques,
    Universit\'e Paris XIII, 99 avenue J-B Cl\'ement, F-93430 Villetaneuse, France,
    \quad\emph{Email:} {\tt yueyun@math.univ-paris13.fr}\\

\noindent
\textbf{Davar Khoshnevisan.} Department of Mathematics,  University of Utah,
    155 South 1440 East, JWB 233, Salt Lake City,  Utah 84112-0090, USA,
    \\\emph{Email:} {\tt davar@math.utah.edu}

\end{small}

\end{document}